%% file: OTS_paper.tex
\documentclass[11pt]{article}

\usepackage[normalem]{ulem}

\usepackage{algorithm} 
\usepackage{algpseudocode} 

\usepackage{authblk}
\usepackage{geometry}
\usepackage{enumerate}
\usepackage{latexsym,booktabs}
\usepackage{amsmath,amssymb}
\usepackage{graphicx, color}
\graphicspath{ {./images/} }
\usepackage[singlespacing]{setspace}
\usepackage{optidef}
\usepackage{comment}
\usepackage{stackengine}
\usepackage[toc,page]{appendix}
\usepackage[table,xcdraw]{xcolor}
\usepackage[utf8]{inputenc}
\usepackage{amsthm}
\usepackage[T1]{fontenc}
\usepackage{nomencl}
\usepackage{etoolbox}
\usepackage{rotating}
\usepackage{caption}
\usepackage{subcaption}
\renewcommand\nomgroup[1]{%
  \item[\bfseries
  \ifstrequal{#1}{A}{Sets}{%
  \ifstrequal{#1}{B}{Variables}{%
  \ifstrequal{#1}{C}{Parameters}{}}}%
]}
\usepackage{tikz}
\usetikzlibrary{shapes.geometric, arrows}
\usepackage{xfrac}

\tikzstyle{startstop} = [rectangle, rounded corners, minimum width=3cm, minimum height=1cm,text centered, draw=black, fill=red!30]
\tikzstyle{io} = [trapezium, trapezium left angle=70, trapezium right angle=110, minimum width=3cm, minimum height=1cm, text centered, draw=black, fill=blue!30]
\tikzstyle{process} = [rectangle, minimum width=2cm, minimum height=0.5cm, text centered, text width=4cm, draw=black, fill=orange!30]
\tikzstyle{block} = [rectangle, draw, fill=orange!30, text width=8em, text centered, rounded corners, minimum height=2em]
\tikzstyle{decision} = [diamond, minimum width=2cm, minimum height=1cm, text centered, text width=2cm, draw=black, fill=green!30]
\tikzstyle{arrow} = [thick,->,>=stealth]

\tikzstyle{vertex}=[circle,fill=black!25,minimum size=20pt,inner sep=0pt]
\tikzstyle{selected vertex} = [vertex, fill=red!24]
\tikzstyle{edge} = [draw,thick,->]
\tikzstyle{weight} = [font=\small]
\tikzstyle{selected edge} = [draw,line width=5pt,-,red!50]
\tikzstyle{ignored edge} = [draw,line width=5pt,-,black!20]

\makenomenclature

\newtheorem{Proposition}{Proposition}

\numberwithin{Theorem}{section}
\numberwithin{Definition}{section}
\numberwithin{Lemma}{section}
\numberwithin{Algorithm}{section}
\numberwithin{equation}{section}

\title{Data-driven Heuristics for DC optimal transmission switching problem}
\author[1]{Juncheng Li\thanks{Email: j.li83@lancaster.ac.uk}}
\author[2]{Trivikram Dokka\thanks{Email: t.dokka@qub.ac.uk}}
\author[1]{Guglielmo Lulli\thanks{Email: g.lulli@lancaster.ac.uk}}
\author[3]{Fabrizio Lacalandra\thanks{Email: flacalandra@arera.it}}
\affil[1]{Department of Management Science, Lancaster University, United Kingdom}
\affil[2]{Department of Management, Queens University Belfast, United Kingdom}
\affil[3]{The Italian Regulatory Authority for Energy, Networks and Environment (ARERA), Italy}

\begin{document}
\maketitle

\begin{abstract}

The goal of Optimal Transmission Switching (OTS) problem for power systems is to identify a topology of the power grid that minimizes the cost of the system operation while satisfying the operational and physical constraints. Among the most popular methods to solve OTS is to construct approximation via integer linear programming formulations, which often come with big-M inequalities. These big-M inequalities increase, considerably, the difficulty of solving the resulting formulations. Moreover, choosing big-M values optimally is as hard as solving OTS itself. In this paper, we devise two data-driven big-M bound strengthening methods which take network structure, power demands and generation costs into account. We illustrate the robustness of our methods to load changes and impressive runtime improvements of mixed-integer solvers achieved by our methods with extensive experiments on benchmark instances.  The speedup by one of the proposed methods is almost 13 times with respect to the exact method.
\end{abstract}
\newpage

\mbox{}
\nomenclature[A]{$V$}{set of buses in the network}
\nomenclature[A]{$E$}{set of transmission lines in the network}
\nomenclature[A]{$G$}{set of generators}

\nomenclature[B]{$p$}{active power production}
\nomenclature[B]{$\theta$}{voltage angle at bus}
\nomenclature[B]{$f$}{active power flow across transmission line }
\nomenclature[B]{$x$}{binary variable representing on/off status of transmission line }
\nomenclature[C]{$c$}{marginal cost of power production}
\nomenclature[C]{$M$}{big-M value}
\nomenclature[C]{$d$}{demand at bus}

\nomenclature[C]{$B$}{susceptance of transmission line}
\nomenclature[C]{$\bar{f}$}{maximum capacity of transmission line }

\section{Introduction}
The economical dispatch of electric power on grids has been attracting the attention of both academia and Operations Research (OR) practitioners in the energy sector for decades. The optimal power flow (OPF) problem is to determine a minimum-cost delivery scheme of electric power to meet power demand subject to physical laws - such as Kirchhoff's Laws and Ohm's Laws - as well as other operational limits imposed by the grid. The OPF problem is a non-convex quadratic constrained quadratic programming problem first proposed by Carpentier \cite{carpentier1979optimal} and studied since then. While the OPF problem treats the configuration of the grid as fixed, it is now a common practice to temporarily switch off some branches in specific operating hours \cite{hedman2011review}.  

Economic benefits from switching off branches of a grid can be tremendous, as it enables system operators to  lower significantly the overall production costs and/or balance the load on the grid \cite{fisher2008optimal, hedman2008optimal, hedman2011optimal}. Indeed, the optimal unit commitment schedule can change by modifying the grid topology \cite{hedman2010co}. Equally important, is that the use of switching in grids may facilitate the integration of renewable power into the system  \cite{shi2017stochastic}. The problem of finding the optimal grid topology, i.e., the topology that allows meeting the energy demand at the lowest generation costs, is known in the literature as Optimal Transmission Switching (OTS). The name derives from the fact that the optimal topology is obtained by switching off some branches. From the mathematical modelling point of view, the OTS problem is formulated by adding the switching component of branches to the OPF model. This leads to 
a mixed integer non-convex nonlinear programming formulation that poses an extraordinary computational challenge. Under some realistic operational assumptions, the OTS formulation is often approximated as a mixed integer linear program (MILP) with big-M parameters  \cite{fisher2008optimal}, to take advantage of mature MIP solvers. Such a MILP approximation - which is still NP-hard \cite{lehmann2014complexity}\cite{kocuk2016} - is referred in the literature as the DC OTS problem.  Given the computational complexity of the problem, a  number of heuristic methods have been proposed in the literature. Fuller et al. \cite{fuller2012fast} presented two greedy algorithms that iteratively switch off branches guided by linear programming duality information. Ruiz et al. \cite{ruiz2012tractable} analyzed three different policies to remove braches in a greedy type algorithm.   Prescreening heuristics to reduce the set of switchable arcs according to properties of the electric power system were proposed in \cite{barrows2012computationally, liu2012heuristic, wu2013selection}. Johnson et al. \cite{johnson2020k} studied a k-nearest-neighbor clustering approach to compute DC OTS solution making use of historical information. On the other side of the methodological spectrum, i.e, exact methods, Kocuk et al.\cite{kocuk2016} gave a polyhedral study of the problem and proposed cycle-based valid inequalities to strengthen the MILP formulation. 

In this work, we propose to compute solutions of the OTS problem using MIP solvers. However, to exploit the capabilities of these solvers, the values of the big M parameters have to be carefully set; and this is the objective of this paper. In fact, while small big-M parameters result in strong linear relaxation that speeds up solution procedure in MIP solvers, undersized ones can cut off the optimal solution. Binato et al. \cite{binato2001new} observed that elementary weighted longest paths between the endpoints of each branch of the grid can be used to define the big-M parameters for the grid expansion planning problem - a problem that shares similar characteristics with DC OTS. Indeed, we show in \S \ref{sec::bigM} that big-M parameters determined by elementary weighted longest paths reserves at least one optimal solution of the DC OTS problem. However, longest-path big-M parameters may be too conservative (big) because they preserve all the feasible region without taking any power demand and production cost into account. This motivates the development of methods to compute smaller big-M parameters. However, the problem of computing the smallest big-M parameters that do not cut off any feasible solution of the DC OTS problem is NP-hard and it is as difficult as solving the DC OTS problem, see  \cite{fattahi2018bound}. This prompts the development of heuristic methods to compute big-M parameters for DC OTS problem, which is absent in the existing literature. In this work, we propose two polynomial-time heuristic methods to compute big-M parameters for DC OTS problem by solving a sequence of linear programming problems.
The rest of the paper is organized as follows: \S \ref{sec::bigM} provides a brief description of DC OTS problem as well as an example motivating the use of input data to compute small big-M values. \S \ref{sec::Heuristic} illustrates the 2 proposed heuristic methods to compute big-M values. \S \ref{sec::Comp} shows computational experiments comparing the proposed heuristic methods and the exact method. \S \ref{sec::Conc} draws the conclusion of this work and mentions possible future research.

\section{Big-M values and motivating example}
\label{sec::bigM}

For the sake of completeness, we first present the mathematical programming formulation (\ref{eq:DC_OTS}) of the DC OTS problem. The interested reader may refer to \cite{fisher2008optimal} for further details.
The grid is represented by a graph $D = (N,E)$, where $N$ denotes the set of buses (nodes) and $E$ denotes the set of branches (edges) linking the buses. $G \subset N$ is the set of buses connected to power generators. Let the positive $f_{ij}$ denote the power flow from bus $i$ to bus $j$. The negative value of  $f_{ij}$ denotes the power flow in the opposite direction (from bus $j$ to bus $i$).

\begin{subequations} \label{eq:DC_OTS}
\begin{align}
    \min_{p,f,\theta, x} \quad & \sum_{i \in G}p_i \cdot c_i \label{eq:DC_OTS_obj}\\
    \textrm{s.t.}\quad & \sum_{j \in V}f_{ji} - \sum_{j \in V}f_{ij} = d_i, \; \forall i \in N\setminus G \label{eq:DC_OTS_power_balance}\\
    & \sum_{j \in V}f_{ji} - \sum_{j \in V}f_{ij} + p_i = d_i, \;\forall i \in G \label{eq:DC_OTS_power_balance_gen}\\
    & f_{ij} = B_{ij}(\theta_i - \theta_j)\cdot x_{ij}, \;\forall (i,j) \in E  \label{eq:kirchhoff_dc_switch} \\ 
    & -\bar{f_{ij}}\cdot x_{ij} \leq f_{ij} \leq \bar{f_{ij}}\cdot x_{ij}, \; \forall (i,j) \in E \label{eq:DC_OTS_power_flow_limit}\\
    & p_i^{min}\leq p_i \leq p_i^{max}, \; \forall i \in G \label{eq:DC_OTS_power_prod_limit}\\
    & x_{ij} \in \{0,1\}, \; \forall (i,j) \in E. \label{eq:x_bin}
\end{align}
\end{subequations}
The objective function (\ref{eq:DC_OTS_obj})  minimizes the power production costs. Constraints (\ref{eq:DC_OTS_power_balance}) and (\ref{eq:DC_OTS_power_balance_gen}) ensure that net power flow into each bus equals its power demand ($d_i$), where $p_i$ is the power provided by generator {\it i}. (\ref{eq:kirchhoff_dc_switch}) imposes that power flow across branch $(i,j)$ equals the voltage angle difference ($\theta_i - \theta_j$) multiplied by its susceptance ($B_{ij}$), if branch $(i,j)$ is on service (i.e., $x_{ij}=1$. $x_{ij}$ is binary variable representing on/off status of the branch). Constraints (\ref{eq:DC_OTS_power_flow_limit}) and (\ref{eq:DC_OTS_power_prod_limit}) impose limits on power flow ($\bar{f_{ij}}$) on each branch and on power production ($p_i^{min}, p_i^{max}$)  at each bus {\it i} of $G$.
The bilinear constraint (\ref{eq:kirchhoff_dc_switch}) can be linearized with big-M parameters using the following inequalities:
\begin{equation}
    \label{eq:kirchhoff_dc_switch_linearized}
    -(1-x_{ij})M_{ij} \leq f_{ij} - B_{ij}(\theta_i - \theta_j) \leq (1-x_{ij})M_{ij}, \\ \forall (i,j) \in E
\end{equation} 

$M_{ij}$ is the big M value, i.e., a "large enough" number such that constraint (\ref{eq:kirchhoff_dc_switch_linearized}) is redundant when $x_{ij} = 0$.(i.e. implied by other constraints), thus equivalent to non-existent. We here show that $M_{ij}$ computed by elementary weighted longest path from node $i$ to node $j$ does not cut off any feasible solution under realistic assumption.

\begin{Proposition} \label{pro:connectivity_opt}
\cite{kocuk2016} In the DC OTS problem, there exists an optimal solution in which the lines switched on form a connected network.
\end{Proposition}

It is also desirable for electric grid to connected, see \cite{jabr2012minimum}. In the remaining of this paper, we assume solution of DC OTS problem is always connected.

\begin{Proposition}
 \label{pro:path_big_M}
In the DC OTS problem, if a elementary path $P$ from node $i$ to node $j$ is on service, then weight of $P$ times $B_{ij}$ is a valid $M_{ij}$.
\end{Proposition}
\begin{proof}
As the path $P$ is on service, $x_{uv} = 1$ for every edge $(u,v)\in P$. Thus, for each edge $(u,v)\in P$, $-\bar{f}_{uv}\leq B_{uv}(\theta_u - \theta_v) \leq \bar{f}_{uv}$. Therefore, $-\sum_{(u,v) \in P}\frac{\bar{f}_{uv}}{B_{uv}}\leq \sum_{(u,v) \in P}(\theta_u - \theta_v) \leq \sum_{(u,v) \in P}\frac{\bar{f}_{uv}}{B_{uv}}$. As $\theta_i - \theta_j = \sum_{(u,v) \in P}(\theta_u - \theta_v)$, this implies that $-B_{ij}\sum_{(u,v) \in P}\frac{\bar{f}_{uv}}{B_{uv}}\leq B_{ij}(\theta_i - \theta_j) \leq B_{ij}\sum_{(u,v) \in P}\frac{\bar{f}_{uv}}{B_{uv}}$. Thus, constraint (2.8) is redundant when $x_{ij}=0$, with $M_{ij} := B_{ij}\sum_{(u,v) \in P}\frac{\bar{f}_{uv}}{B_{uv}}$.
\end{proof}

Proposition \ref{pro:path_big_M} shows that $M_{ij}$ equalling $B_{ij}$ times the weight of the longest elementary path from node $i$ to node $j$ does not cut off any feasible solution, where weight of a path $P$ is defined as $\sum_{(u,v) \in P} \frac{\bar{f}_{uv}}{B_{uv}}$. We refer to this big M value as longest-path big-M value. 

The choice of $M_{ij}$ parameters influences significantly the performance of MIP solvers based on LP branch-and-cut procedure. Large values of $M_{ij}$'s often result in ``weak" LP relaxations thus hampering the performance of MIP solvers. Indeed, in the formulation it is desirable to have the smallest big M values that do not cut off the optimal solution.
Fattahi et al. \cite{fattahi2018bound} define the theoretically smallest $M_{ij}$ as 
\begin{equation}
    M_{ij}^{opt} = B_{ij} \times \max_{(p,f,\theta, x)\in F}\{|\theta_i - \theta_j|\}
\end{equation}
where $F:=\{(p,f,\theta, x): \eqref{eq:DC_OTS_power_balance}-\eqref{eq:x_bin}\}$. They also showed that {\it i)} computing $M_{ij}^{opt}$ is NP-hard and there is no polynomial-time constant-factor approximation algorithm for $M_{ij}^{opt}$; and {\it ii)} $M_{ij}^{opt}$ can be arbitrarily large on networks with special topology. 
But even using the weighted longest path between the endpoints of each edge of the network as big M value, as suggested by Binato et al. \cite{binato2001new} among others, the computation of such values is still demanding at least from the theoretical point of view. Moreover, the weighted longest past is based exclusively only on the network topology and ignores completely the power demand at each bus and power production costs at each power plant, which may lead to too conservative values (larger than $M_{ij}^{opt}$). Power demand, generation capacity and production costs are all relevant information that can be used to compute smaller big M values, as shown in the example below.

Consider the example depicted in Figure \ref{fig:big_M_exa}. Suppose that the links marked with a thicker line are switchable, while the others are not. The number next to the link is the transmission capacity $\bar{f}$ and the susceptance $B$ is 1 for all the links. In this case, the big-M parameters suggested by (weighted) longest path are equal to 6 for both the switchable links. Suppose now that the generator $g1$ is cheaper and there is a power demand of 2.5 MW at bus {\it t} of the network. The power demand at the other buses is null. In this case, a big-M value of 1 for link $(f,g2)$ is sufficient. To see this note that in the event $(b,g1)$ is switched on a maximum of 1.5 units can be supplied from $g1$, this implies a maximum of 1 units is supplied from $g2$. In fact even $M_{f,g2}=0.5$ is sufficient as no unit is required from $g2$ if $(b,g1)$ is switched off. This example suggests taking network structures, power demand and generation costs into consideration to get smaller big M parameter. 

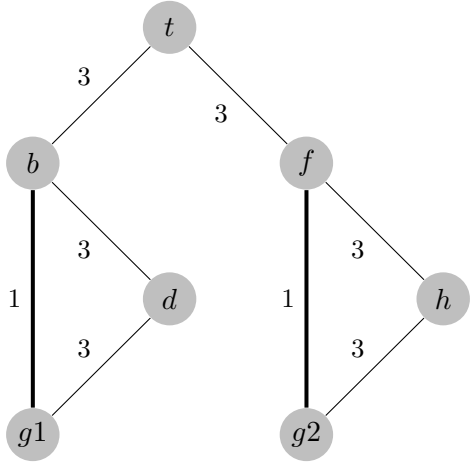
\begin{figure}[H]
\centering
\begin{tikzpicture}[scale=1.8, auto,swap]
    \foreach \pos/\name in { {(2,1)/t},
                            {(3,0)/f},  {(2,-1)/d}, {(1,0)/b},  {(4,-1)/h},
                            {(1,-2)/g1},  {(3,-2)/g2}}
        \node[vertex] (\name) at \pos {$\name$};
     \foreach \source/ \dest /\weight in { 
                                         t/b/3, t/f/3, f/h/3, h/g2/3, b/d/3, d/g1/3}
        \path (\source) edge [-]  node[weight] {$\weight$} (\dest);
        \foreach \source/ \dest /\weight in { 
                                         b/g1/1, f/g2/1}
        \path (\source) edge [-, line width = 1.5]  node[weight] {$\weight$} (\dest);
        
\end{tikzpicture}
\caption{Big-M Example (Switchable links - highlighted)}
\label{fig:big_M_exa}
\end{figure}

\section{Heuristic methods for big M values}
\label{sec::Heuristic}
In this section, we present two data driven polynomial-time heuristic methods to compute big M parameters. Both the  methods take power demands and production costs into account. We call the first method {\it k-shortest-path}  (kSP) and the second {\it  k-nearest-neighborhood simulation method} (kNN).

\subsection{k-shortest-path method}

The k-shortest-path method is based on the observation that follows directly from Proposition \ref{pro:path_big_M}. Indeed, the phase angle difference between any pair of nodes -and hence the big M value - is limited above by the weighted shortest path between the two nodes, where the weight of each edge is given by  $\frac{\bar{f}}{B}$. If one edge of the shortest path is removed from the graph (i.e., it is switched off), than the second shortest path becomes the one limiting the phase angle difference between the two nodes. We can iterate the procedure, i.e.,  removing one edge from the current shortest path and selecting the next shortest path, until we either reach the k-shortest-path (for a given parameter k) or we obtain an infeasible problem, meaning that no edge of the path can be switched off. In other words, this path has to {\it be on service}, and its weight determines the big M value. Given a path, we use the approach proposed by Fuller et al. \cite{fuller2012fast} to select the edge to be switched off. Fueller et al. used DC OPF duality information and KKT optimality conditions to compute edges' shadow price, which provides a (local) measure of objective function's improvement if the edge is switched off.  More in particular, they showed that the shadow price $\alpha_{ij}$ of edge ({\it i,j}) is given by $(\pi_i - \pi_j)f_{ij}$, where $\pi_i$ and $\pi_j$ are the dual variable of the power balance constraint at buses {\it i} and {\it j} of the OPF problem respectively. These constraints are exactly the same as constraints \eqref{eq:DC_OTS_power_balance} and \eqref{eq:DC_OTS_power_balance_gen} of the DC OTS formulation described above. 

A negative value of $\alpha$ "infers" that removing edge $(i,j)$ from the network would results in lower production costs. Therefore, for each shortest path, we iteratively check whether an edge can be off service without making DC OPF infeasible from the smallest $\alpha$ to the largest one. We also test edges with $\alpha > 0$, because we are only interested in verifying if the path is a {\it be-on-service } path or not, i.e., if the underlying OPF problem - after disconnecting the path - is feasible or not. 

Once we find that an edge can be switched off, we remove this edge from the graph and proceed to the next shortest path.
If none of the edges in the path can be switched off (i.e. removing any of those edges makes the reduced OPF problem infeasible), we attempt to restore feasibility of the reduced OPF problem by removing one more edge outside the path that contributes to infeasibility. If feasibility of the reduced OPF problem is successfully retrieved in this way, it shows that path under consideration can be off-service and the algorithm proceed to consider the next shortest path. Otherwise, we label the path under consideration as "must-on-service" and outputs the big M value determined by weight of this path. In principle, we can attempt to restore feasibility of the infeasible reduced DC OPF problem by removing more than one edge. However, that leads to significantly heavier computation burden thus we choose not to do so. 

\begin{subequations} \label{eq:dc_OPF}
\begin{align}
    \min_{p,f,\theta} \quad & c^T \cdot p \\
    \textrm{s.t.}\quad & \sum_{j \in V}f_{ji} - \sum_{j \in V}f_{ij} + p_i = d_i, \;\forall i \in N  \label{eq:cons_flow_balance}\\
    & f_{ij} = B_{ij}(\theta_i - \theta_j), \;\forall (i,j) \in E  \label{eq:cons_kirchhoff_dc} \\ 
    & -\bar{f_{ij}} \leq f_{ij} \leq \bar{f_{ij}}, \; \forall (i,j) \in E\\
    & p_i^{min}\leq p_i \leq p_i^{max}, \; \forall i \in G
\end{align}
\end{subequations}

\begin{subequations} \label{eq:dc_OPF_fuller}
\begin{align}
    \min_{p,f,\theta} \quad & c^T \cdot p & & \\
    \textrm{s.t.}\quad & \sum_{j \in V}f_{ji} - \sum_{j \in V}f_{ij} + p_i = d_i, \;& \forall i \in N &\qquad [\pi_i] \label{eq:cons_flow_balance_fuller}\\
    & f_{ij} = B_{ij}\cdot x_{ij} \cdot(\theta_i - \theta_j), \;&\forall (i,j) \in E &\qquad [\sigma_{ij}] \label{eq:cons_kirchhoff_dc_fuller} \\ 
    & 1 - x_{ij} = 0, \;&\forall (i,j) \in E &\qquad [\alpha_{ij}] \label{eq:fuller_lambda}\\
    & -\bar{f_{ij}}\cdot x_{ij} \leq f_{ij} \leq \bar{f_{ij}} \cdot x_{ij}, \; &\forall (i,j) \in E &\\
    & p_i^{min}\leq p_i \leq p_i^{max}, \;& \forall i \in G &
\end{align}
\end{subequations}

In order to find the edges contributing to infeasibility, we use a method inspired by the method of identifying minimal (or irreducible) infeasible subsystem (MIS or IIS) of linear inequalities in \cite{codato2006combinatorial}. Consider an arbitrary infeasible set of linear equalities and linear inequalities 
\begin{equation} \label{eq:infeasible_linear_system}
    \{x:Ax = b, Hx \geq g\}
\end{equation}
A minimal infeasible subsystem of (\ref{eq:infeasible_linear_system}) is a inclusion-minimal set of linear equalities and linear inequalities selected from $Ax=b$ and $Hx \geq g$, respectively:
\begin{equation}\label{eq:infeasible_linear_subsystem}
     \{x:A'x = b', H'x \geq g'\}
\end{equation}
such that it has no feasible solution (i.e. (\ref{eq:infeasible_linear_subsystem}) is empty). In order to find a feasible subsystem of (\ref{eq:infeasible_linear_system}), we must remove at least one linear equality or inequality from each of its MIS.

In the DC OPF problem (\ref{eq:dc_OPF}), removing an edge $(i,j)$ from the network is equivalent to removing the constraint (\ref{eq:cons_kirchhoff_dc}) corresponding to the edge $(i,j)$ and imposing an additional constraint $f_{ij}=0$. Let us consider the constraints (\ref{eq:cons_kirchhoff_dc}) in a MIS of the OPF problem (\ref{eq:dc_OPF}) and the set of edges $E_{MIS}$ corresponding to those constraints. OPF problem (\ref{eq:dc_OPF}) on switchable network can be feasible only if at least one edge in $E_{MIS}$ is switched off. Thus, we restrict the candidates for $e_{out}$ to $E_{MIS}$.
Next we describe implementation details of this method. Let us use $n$ to represent the number of nodes, $m$ to represent the number of edges, $v_n$ to represent a vector $v$ with $n$ entries and $V_{n\times m}$ to represent a matrix $V$ with $n \times m$ entries. For the ease of presentation, in the remaing of this work, we will present the OPF problem (\ref{eq:dc_OPF}) in a compact matrix-vector form $\min_x\{\hat{c}^T x: \hat{A}x=\hat{b},\hat{H}x \geq \hat{g}\}$, where $\hat{c} = \begin{pmatrix} c_n\\0_n\\0_m\end{pmatrix}$ , $x=\begin{pmatrix} p_n\\ \theta_n \\ f_m
\end{pmatrix}$, $\hat{A} = \begin{pmatrix} I_{n\times n} & 0_{n \times n} & \Omega_{n \times m}\\0_{m\times n} & B_{m\times n} & -I_{m\times m}\end{pmatrix}$, $\hat{H}=\begin{pmatrix} 0_{m\times n} & 0_{m\times n} & I_{m\times m}\\0_{m\times n} & 0_{m\times n} & -I_{m\times m}\\
I_{n\times n} & 0_{n\times n} & 0_{n\times m} \\-I_{n\times n} & 0_{n \times n} & 0_{n \times m}\end{pmatrix}$, $\hat{b}=\begin{pmatrix}
d_n \\ 0_m\end{pmatrix}$ and $\hat{g}=\begin{pmatrix}-\bar{f}_m\\ -\bar{f}_m \\ \underline{p}_n\\-\bar{p}_n\end{pmatrix}$. $I_{n \times n}$ means $n$-dimensional identity matrix, $\Omega_{n\times m}$ denotes the node-arc incidence matrix, matrix $B_{m\times n}$ is a sparse matrix with non-zero components: $B_{l,i}=B_{ij}$ and $B_{l,j}=-B_{ij}$ for each edge $l$ starting from node $i$ and ending at node $j$. To identify $e_{out}$, we first remove the edge with the smallest $\alpha$ value in the path from the network. Let us denote the remaining constraints in (\ref{eq:dc_OPF}) as $\{\hat{A}x=\hat{b},\hat{H}x \geq \hat{g}\}$ and construct the LP $\min_x \{0^T x: \hat{A}x = \hat{b}, \hat{H}x \geq \hat{g}\}$ as well as its dual $\max_{y,u} \{y^T \hat{g}+ u^T \hat{b}: y^T \hat{H} + u^T \hat{A} = 0^T, y \geq 0\}$ with $u$ being the dual variables of equality constraints and $y$ being the dual variables of inequality constraints. Since dual infeasibility is excluded by the solution $(y,u)=0$, primal infeasibility corresponds to dual unboundedness, i.e. there exists a dual solution $(y,u)$ such that $y^T \hat{g}+ u^T \hat{b} > 0$. Therefore, we can normalize the dual solution by replacing the objective with a constraint $y^T \hat{g}+ u^T \hat{b}=1$. It is shown in \cite{gleeson1990identifying} that the non-zero elements of each vertex of the polyhedron 
\begin{equation} \label{eq:poly_find_MIS}
    \{(y,u):y^T \hat{H} + u^T \hat{A} = 0^T, y^T \hat{g}+ u^T \hat{b}=1, y \geq 0\}
\end{equation}
defines a MIS of (\ref{eq:infeasible_linear_system}) and thus a candidate set for $e_{out}$. In order to find a vertex of (\ref{eq:poly_find_MIS}), we minimize $1^T y$ over (\ref{eq:poly_find_MIS}). This LP is guaranteed to be solved to optimality. We then use the set of edges corresponds to non-zero $u_i, i \in J$ in the optimal solution as the candidates for $e_{out}$, where $J$ is the set of indices of constraints (\ref{eq:cons_kirchhoff_dc}). We then attempt to select an edge from the candidates to be $e_{out}$ as shown in Algorithm \ref{alg::kSP}. If no edge in the candidate set satisfies the requirement of $e_{out}$, we decide that the current path "must" be on service in an economical network configuration, which determines the k-shortest-path big-M values. A formal description of kSP method is given in Algorithm \ref{alg::kSP}, where "dualMIS" denotes the problem of minimizing $1^T y$ over (\ref{eq:poly_find_MIS}) and $w(P^k_{ij})$ denote weight of $k$-th shortest path from node $i$ to node $j$.

\begin{algorithm}
	\caption{- kSP} 
	\begin{algorithmic}[1]
		\State Set parameters: $k_{max}$, $e_{max}$, $l$
		\For {$(i,j) \in E$}
		    \State Solve DC OPF on G
		    \State $\alpha^0 \gets \alpha^0(G)$  
		    \State $k \gets 1$
		    \While {$k < k_{max}$}
		        \If{all edges in $k$-th shortest path is on service}
		            \State Order edges in $k$-th shortest path from smallest $\alpha^{k-1}$ to largest: $\{j_1,\ldots,j_m\}$
		            \State $i \gets 0$
		            \While{$i < e_{max}$}
		                \State $i \gets i+1$
		                \State Solve DC OPF on $G^\prime = G(N; E \setminus j_i)$
		                \If {DC OPF feasible}
		                \State $\alpha^k \gets \alpha^k(G')$
		                \State $E \gets E \setminus j_i$
		                \State Go to line \ref{update_k}
		                \EndIf
		            \EndWhile
		            \State Solve dualMIS on $G^\prime = G(N; E \setminus j_1)$
		            \State Compute $E^\prime = \{(p,q)\in E: u_{pq} \neq 0\}$
		            \State Compute $E^{\prime \prime} = \{h_1,\ldots,h_s\}$ by ordering edges in $E^\prime$ from smallest $\alpha^{k-1}$ to largest
                    \State $v \gets 0$
                    \While{$v<s$}
                        \State $v \gets v + 1$
                        \State Solve DC OPF on $\hat{G} = G(N; E \setminus \{j_1,h_v\})$
                        \If{DC OPF feasible}
		                \State $\alpha^k \gets \alpha^k(\hat{G})$ 
		                \State $ E \gets E \setminus \{j_1,h_v\}$
		                \State Go to line \ref{update_k}
                        \EndIf
                    \EndWhile
                    \State Go to line \ref{output_M}
		        \EndIf
                \State $k \gets k + 1$ \label{update_k}
		    \EndWhile
		    \State $k \gets k + l$ \label{output_M}
		    \State $ M_{ij} =  w(P_{ij}^k) \times B_{ij}$
		\EndFor
	\end{algorithmic} 
	\label{alg::kSP}
\end{algorithm}

\subsection{k-nearest-neighbor Simulation Method}
Compared to \textit{k-SP}, our second heuristic is a randomized method that relies on the following intuition. As illustrated in Figure \ref{fig:big_M_exa}, the $M_{ij}$ value depends on which edges remain in the network. More importantly, the value of $M_{ij}$ is likely to depend on the status of switchable edges that are in close vicinity rather than those which are spatially far. This is because switching off nearby edges likely removes shorter paths forcing a larger $M_{ij}$ value as compared to faraway edges. We use this observation to explore the impact on $M_{ij}$ by switching off randomly selected edges in a  given neighbourhood. Moreover, this heuristic is inspired by large class of randomized local search heuristics often used in large scale combinatorial optimization such as \textit{GRASP} heuristics \cite{feo1995}. A formal description of the \textit{k-nearest-neighbor (kNN) simulation} method is given below in Algorithm \ref{alg::kNN}.

\begin{algorithm}
	\caption{- kNN} 
	\begin{algorithmic}[1]
		\State Set parameters: $k$, $h<1$, $s$, $r$
		\For {$(i,j) \in E$}
    		\State $inf \leftarrow 0$
    		\State Compute $N_{ij}= \{ (u,v) \in E: d_{ij}(u) \leq k \land d_{ij}(v) \leq k\}$
			\For {$iter=1,2,\ldots,r$}
        		\State Select $N^\prime_{ij} \subset N_{ij} : |N^\prime_{ij}| = h \cdot|N_{ij}|$
				\State Solve DC OPF on $G^\prime = G(N; E \setminus N^\prime_{ij})$
				\If {DC OPF feasible} 
				 \State   $\Delta\Theta^{iter}_{ij} =| \theta_i - \theta_j|$
				\Else $\hspace{2mm} \Delta\Theta^{iter}_{ij} = 0$ and $inf \leftarrow inf +1$
				\EndIf
			\EndFor
			\State $\Delta\Theta_{ij} \leftarrow max_{iter}\{ \Delta\Theta^{iter}_{ij} \}$ 	
			\If {$inf < r$} 
				\State $M_{ij} = \min \{s \times \Delta\Theta_{ij}, w(P_{ij}^L)\}\times B_{ij}$ 
			\Else  $\hspace{2mm} M_{ij} =  w(P_{ij}^L) \times B_{ij}$
			\EndIf
		\EndFor
	\end{algorithmic} 
	\label{alg::kNN}
\end{algorithm}
At line 4 of the algorithm, to compute the neighborhood of edge $(i,j)$ ($N_{ij}$), i.e., the set of edges that are close to $(i,j)$, we use  Dijktra's algorithm. Indeed, we consider in the neighborhood all the edges whose endpoints are at most $k$ edges away from either {\it i} or {\it j} (in formula, $d_{ij}(u) \leq k$). If the DC OPF problems are all infeasible - line 7 of the algorithms -, we use as big M value for edge $(i,j)$ the value provided by the longest past method $w(P_{ij}^L) \times B_{ij}$; wehre $w(P_{ij}^L)$ is the weight of longest weighted path from node $i$ to node $j$ ($P_{ij}^L$).

\subsection{An illustrative example}
In this section we use a small benchmark instance IEEE 14-bus system \cite{IEEE14} displayed in Figure \ref{fig:14bus}, to illustrate with an example on how longest-path, k-shortest-path and kNN simulation method compute big-M value. For this purpose, we show the process of computing big M value of edge $(3,4)$ -i.e., $M_{3,4}$ - with the listed three methods. Transfer capacity is set to 27 for all the branches, as these values were not provided in the original data. For  kSP and kNN simulation methods, we consider the following parameters' setting: $k_{max} = 5$, $e_{max} = 3$ and $l = 1$ for kSP and $k=2,h=20,s=10$, and $r=30$ for kNN.  The big-M value of edge $(3,4)$ returned by the three methods is 355.5 (longest-path method), 260.9 (kSP method) and 216.4 (kNN simulation method).

In the longest-path method, $M_{3,4}$ is given by $B_{34}$ times the weight of the (weighted) longest path, i.e., path $3\to2\to1\to5\to6\to12\to13\to14\to9\to4$. 

As far as the kSP method, the first shortest path to consider (first iteration) is edge $(3,4)$ itself. We remove edge $(3,4)$ from the network and solve the OPF problem $OPF((3,4))$, which turns out to be infeasible. We attempt to restore the feasibility of the  OPF problem by removing one additional edge $(e_{out})$ from the network.  To identify candidate edges for $e_{out}$, we solve $dualMIS((3,4))$.  We rank this set of edges by linear programming duality as described in \S \ref{sec::bigM}: $\{(4,5), (2,4), (6,11), (1,2), \ldots\}$. We check if the OPF problem becomes feasible by removing one additional edge of the list. The first feasible OPF problem is $OPF((3,4),(1,2))$,  with $e_{out} := (1,2)$. We then remove $(3,4)$ and $(1,2)$ from the network.  At the second iteration, the shortest path (on the current network) is $3\to2\to4$. Using arguments of \S\ref{sec::bigM}, we rank the two edges as follows: $\{(4,2),(2,3)\}$. We solve $OPF((4,2))$, which turns out to be infeasible. We then solve $OPF((2,3))$, which is feasible. We then remove edge $(2,3)$ from the network. At the third iteration,  it turns out that the third, fourth and fifth shortest paths are no longer active as some edges of those paths have already been removed. As $k_{max}$ is reached when considering the fifth shortest path, we terminate the algorithm and output $k_{34} = 5 + l = 6$. Big-M value on edge $(3,4)$ is defined by $B_{34}$ times weight of the sixth shortest path from node $3$ to node $4$.

In the kNN simulation method, first we remove  edge $(3,4)$ from the network and identify the set of nodes whose distance is at most $k=2$ from either node $3$ or node $4$, i.e., $\{1,2,5,7,9,10,14,8\}$. All the edges with both the endpoints in the list define the neighborhood, specifically   
$\{(4,2),(3,2),(2,1),(2,5),$ $(1,5),(4,5),(9,10),(4,9),(7,9),(4,7),(7,8),(9,14)\}$.  Select randomly $h=20\%$ of the edges (i.e., 2 edges in this example) from the neighborhood. Solve OPF problem on the network without these randomly selected (two) edges and compute the phase angle difference $\Delta \theta_{34} = |\theta_4 - \theta_3|$ from the OPF solution. Repeat this procedure for $r=30$ times and denote the largest $\Delta \theta_{43}$ as $\Delta' \theta_{43}$. We set as big M value for edge $(4,3)$ $M_{3,4}=\min \{s \times \Delta \theta_{43}', w(P_{43}^L)\}\times B_{43}$ , where $P_{43}^L$ denote the longest weighted path from node 3 to node 4.

\begin{figure} [h]
     \begin{subfigure}[b]{0.45\textwidth}
        \centering
        \includegraphics[width=\textwidth]{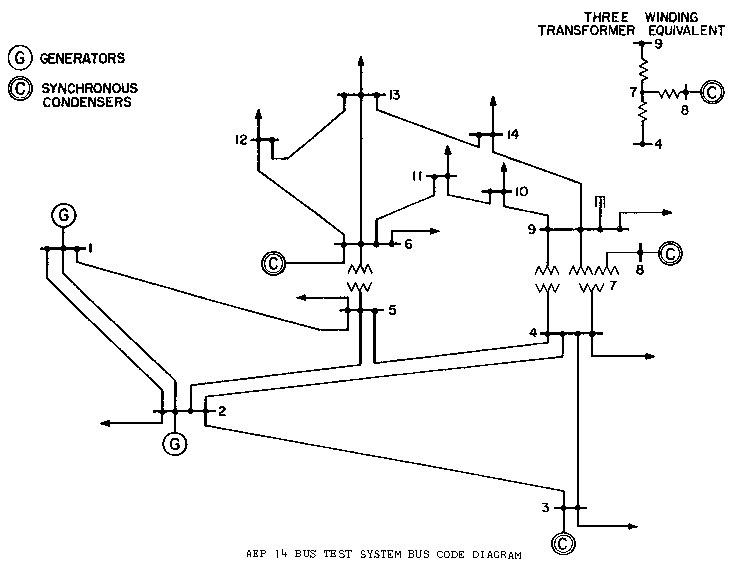}
         \caption{diagram of 14-bus system}
         \label{fig:diag_14_bus}
     \end{subfigure}
     \hfill
     \begin{subfigure}[b]{0.45\textwidth}
         \centering
        \includegraphics[width=\textwidth]{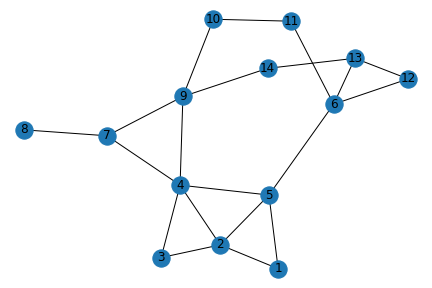}
         \caption{network of 14-bus system}
         \label{fig:graph_14_bus}
     \end{subfigure}
     \caption{IEEE 14 bus instance}
     \label{fig:14bus}
\end{figure}

\section{Computational Experiments}
\label{sec::Comp}
In this section, we present an extensive computational study that demonstrates the effectiveness of our proposed approaches.  
We compare the kSP and kNN methods with the longest-path approach (LWP) on a set of benchmark instances.
More specifically, we consider two IEEE test cases, i.e., case118B and case300B,  widely used in the literature. The former consists of 118 buses, 180 branches and 19 committed generators with different power production costs. The latter includes 300 buses, 409 branches and 61 committed generators with different power production costs. In addition to the nominal demand scenario, which is part of the IEEE instances, we also consider both a low-demand and a high-demand scenario. These demand scenarios are obtained from the nominal one multiplying the load at each bus by 0.95 and 1.05  respectively. 
For each pair of grid and demand scenario - defining a class of instances -, we generated twenty instances by multiplying the load at each bus by a random variable with a uniform distribution in the interval [0.95, 1.05]. Hence, in total we considered 120 instances of the OTS problem with different characteristics. "118/300L", "118/300N", "118/300H" denote the low, normal and high demand class of instances, respectively.
\newline
All the computational experiments are run on a laptop with Intel(R) Core(TM) i7-8750H @2.20 GHz 2.21GHz processor and 16GB RAM.  We use Julia v1.5.2 with JuMP v0.21.5 programming language \cite{lubin2015computing}  to implement the model and the algorithms, and Gurobi 9.0.1 as a solver. To compute the big M values with the longest weighted path approach, the Miller-Tucker-Zemlin (MTZ) longest path formulation \cite{taccari2016integer} is implemented and solved by Gurobi 9.0.1, imposing a time limit of 600 seconds and setting an optimality gap to 1\%.
k-shortest paths are computed using the algorithm proposed in \cite{yen1971finding}.

\subsection{Parameters' tuning and results}
The first phase of the computational study has been devoted to tuning the parameters of both the kSP and kNN methods. The goal is to identify a set of parameters which work well across all instances of a given class.

Table \ref{tab:para} lists all the values of the parameters we have tested for each class of instances (CoI). The ``best" parameters' setting for each class of instances are highlighted with bold fonts.
\begin{table}[htb]
\centering
\begin{tabular}{c | c c c | c c c c} 
 \multicolumn{1}{c}{}    & \multicolumn{3}{c}{kSP} & \multicolumn{4}{c}{kNN} \\ \cline{2-8}
CoI  & $k_{max}$ & $e_{max}$ & $l$  & $k$ & $h$ & $s$ & $r$\\ \hline \hline
118L & 11,{\bf 14},17,20   & 3,{\bf 5},7    & 0,{\bf 1}    & 3,4,{\bf 5} & 5, {\bf 10} & 5,10,{\bf 15} & 50\\
118N & {\bf 11},14,17,20   & 3,{\bf 5},7    & 0,{\bf 1}    & 3,{\bf 4},5 & 5, {\bf 10} & 5,10,{\bf 15}& 50\\
118H & {\bf 11},14,17,20   & 3,{\bf 5},7    & 0,{\bf 1}    & 3,{\bf 4},5 & {\bf 5}, 10 & 5,{\bf 10},15& 50\\ \hline
300L & 15,20,{\bf 25},30   & 5,7,{\bf 9}    & {\bf 0},1    & 4,{\bf 5},6,7 & {\bf 5}, 10 & 5,10,{\bf 15}& 50\\
300N & {\bf 15},20,25,30   & {\bf 5},7,9    & 0,1,2,{\bf 3}        & {\bf4},5,6,7 & {\bf 5}, 10 & 5,10,{\bf 15} & 50\\
300H & {\bf 15},20,25,30   & 5,{\bf 7},9    & 0,1,2,3,4,{\bf 5}    & 4,5,6,{\bf 7} & 5, {\bf 10} & 5,10,{\bf 15}& 50\\ \hline \hline
\end{tabular}
\caption{Parameters' values tested for kSP and kNN.}
\label{tab:para}
\end{table}

Given the nature of our approach, i.e., compute small enough big M values that allow computing (near) optimal OTS solutions in short computational time, quality of the solution and computational time are the two measures of interest. Therefore, we use both the measures to assess the effectiveness of proposed methods.  
To measure the solutions' quality, we use the following ratio $\frac{z_m-z_{LP}}{z_{LP}}$. This ratio - relative gap - compares the objective value  of the solution computed by the proposed method ($z_m$) with the objective value of the solution computed with big M values obtained by the longest-path method ($z_{LP}$). The smaller the relative gap, the better the solution quality. Negative values of the relative gap indicate that the proposed approach provides better solutions than the longest-path method. It is important to recall that all the instances are solved by the solver with a time limit of 600 sec. Therefore, for the sake of accuracy, we are comparing the best solution computed within the imposed time limit, which is the optimal one in many instances. Figure \ref{fig:ksp_tune} display the scatter plots of the kSP method for each class of instances. Each point of the scatter plot represents the average performance of kSP with one setting of the parameters on the twenty instances of the class. It appears that parameter {\it l} is key to the good performance of the method as it influences the trade off between quality of solution and computational time. Indeed this parameter was introduced to ensure a certain degree of robustness of the proposed approach by enlarging the big M values. To larger values of {\it l} correspond better solutions' quality (smaller relative gaps) at the cost of longer computational times. For the instances derived from case118B IEEE, $l=1$ provides a very good balance between solutions' quality and computational time. The relative gaps are very small - negative in several cases - independently of the demand scenario and the setting of the other two parameters. The computational times are also short, especially in the low and high demand scenarios. However, in this case the choice of $k_{max}$ and $e_{max}$ settings is relevant for the time. The settings that guarantee the shortest computational times are the chosen ones or the "best". It is also important to observe, that these settings are quite robust. In particular, the triplet $k_{max}=11$, $e_{max}=5$ and $l=1$  that is best for the nominal and high demand scenario performs quite well also in the low demand scenario (marked on the scatter plots with a star). 

Similar observations are also valid for the case300B IEEE family of instances. However, the class of instances with low demand scenario (300L) turn out to be ``easy" to solve and all the tested parameters perform reasonably well (see scatter plot in Figure \ref{fig:ksp_para_tune_300bus_95}). For the nominal and high demand scenarios, the effect of parameter $l$ becomes evident. Indeed, we observe again data in Figures \ref{fig:ksp_para_tune_300bus_100} and \ref{fig:ksp_para_tune_300bus_105} are clustered by values of $l$.

As far as the kNN method, see Figure \ref{fig:kNN_tune}, it computes good solutions of case118B instances. Indeed, for all the demand scenarios, it computes solutions with (very) small average gap though the computational times are not always very competitive at least in comparison with the kSP method. On the case300B test instances, though the performance with specific settings of the parameters can be good, it appears that several of the tested settings do not provide competitive solutions demonstrating the greater challenge of identifying the right settings to use.

\begin{figure}[h]
     \begin{subfigure}[b]{0.45\textwidth}
         \centering
        \includegraphics[width=\textwidth]{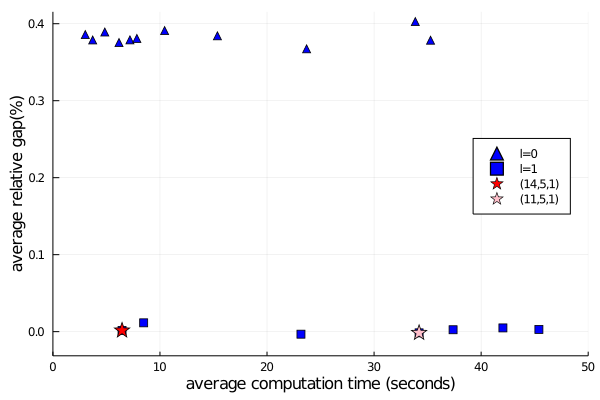}
         \caption{data set: 118L}
         \label{fig:ksp_para_tune_118bus_95}
     \end{subfigure}
     \hfill
     \begin{subfigure}[b]{0.45\textwidth}
         \centering
        \includegraphics[width=\textwidth]{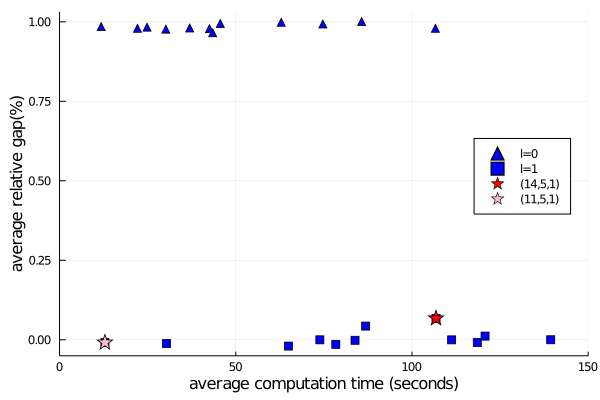}
         \caption{data set: 118N}
         \label{fig:ksp_para_tune_118bus_100}
     \end{subfigure}
    \begin{subfigure}[b]{0.45\textwidth}
         \centering
           \includegraphics[width=\textwidth]{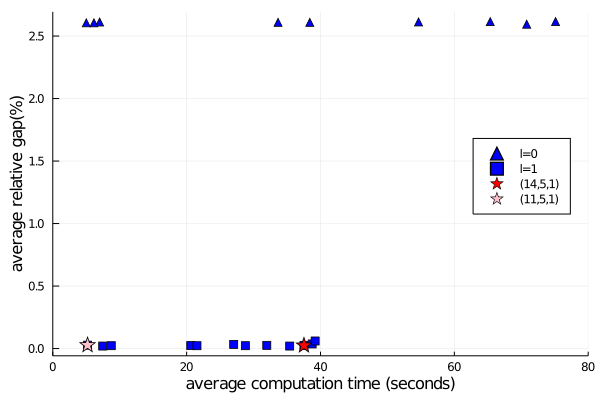}
         \caption{data set: 118H}
         \label{fig:ksp_para_tune_118bus_105}
    \end{subfigure}
     \begin{subfigure}[b]{0.45\textwidth}
         \centering
        \includegraphics[width=\textwidth]{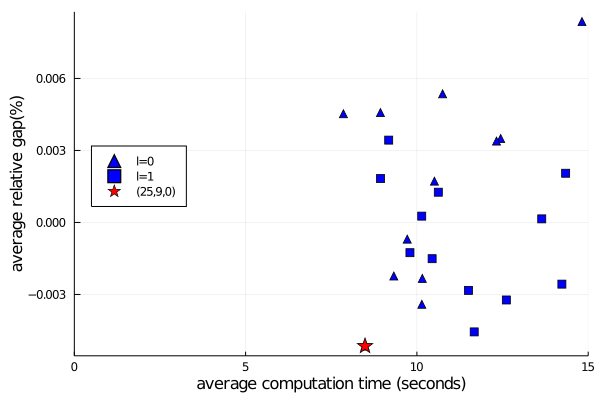}
         \caption{data set: 300L}
         \label{fig:ksp_para_tune_300bus_95}
     \end{subfigure}
     \hfill
     \begin{subfigure}[b]{0.45\textwidth}
         \centering
        \includegraphics[width=\textwidth]{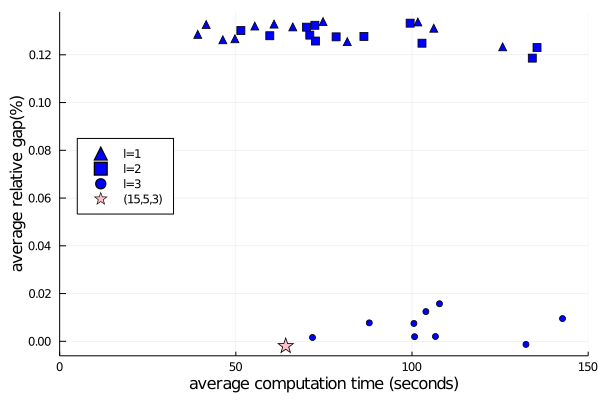}
         \caption{data set: 300N}
         \label{fig:ksp_para_tune_300bus_100}
     \end{subfigure}
      \hfill
    \begin{subfigure}[b]{0.45\textwidth}
         \centering
           \includegraphics[width=\textwidth]{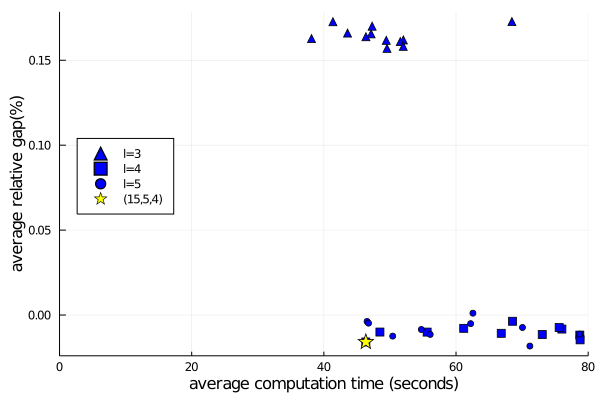}
         \caption{data set: 300H}
         \label{fig:ksp_para_tune_300bus_105}
    \end{subfigure}
    \caption{parameter tuning for kSP method}
    \label{fig:ksp_tune}
\end{figure}    

\begin{figure}[!]
     \begin{subfigure}[b]{0.45\textwidth}
         \centering
        \includegraphics[width=\textwidth]{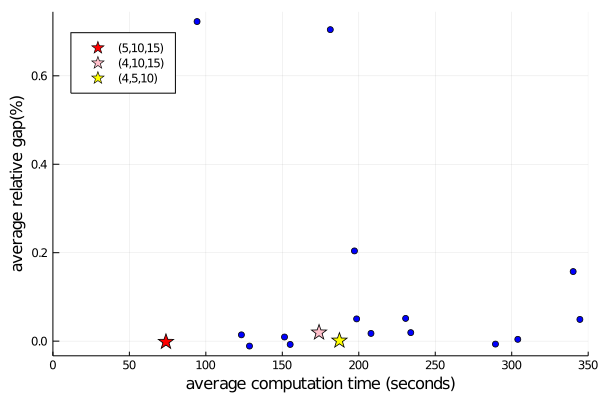}
         \caption{data set: 118L}
         \label{fig:kNN_para_tune_118bus_95}
     \end{subfigure}
     \hfill
     \begin{subfigure}[b]{0.45\textwidth}
         \centering
        \includegraphics[width=\textwidth]{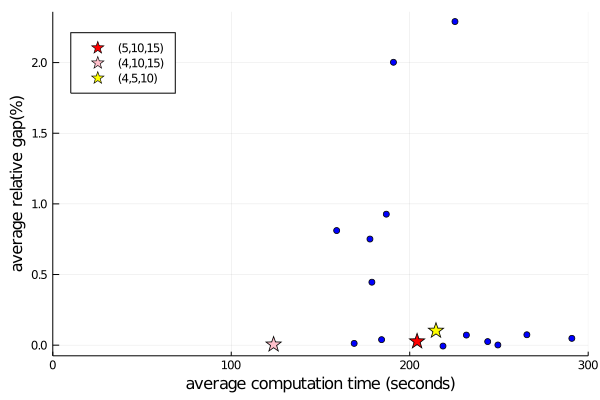}
         \caption{data set: 118N}
         \label{fig:kNN_para_tune_118bus_100}
     \end{subfigure}
    \begin{subfigure}[b]{0.45\textwidth}
         \centering
            \includegraphics[width=\textwidth]{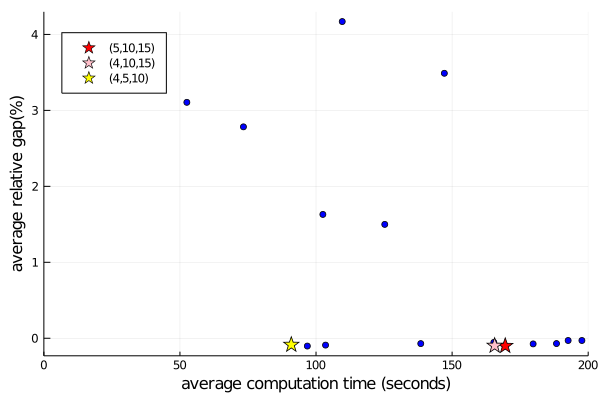}
         \caption{data set: 118H}
         \label{fig:kNN_para_tune_118bus_105}
    \end{subfigure}
     \begin{subfigure}[b]{0.45\textwidth}
         \centering
        \includegraphics[width=\textwidth]{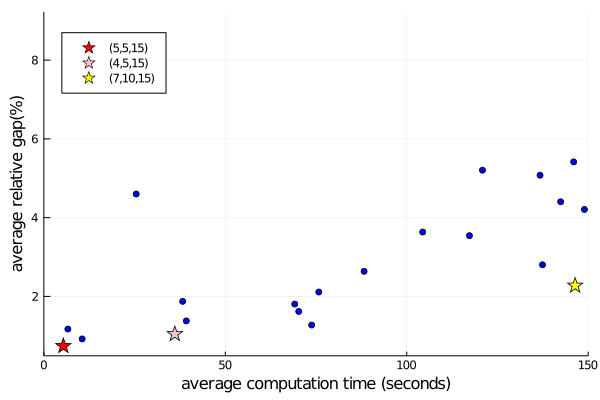}
         \caption{data set: 300L}
         \label{fig:kNN_para_tune_300bus_95}
     \end{subfigure}
     \hfill
     \begin{subfigure}[b]{0.45\textwidth}
         \centering
        \includegraphics[width=\textwidth]{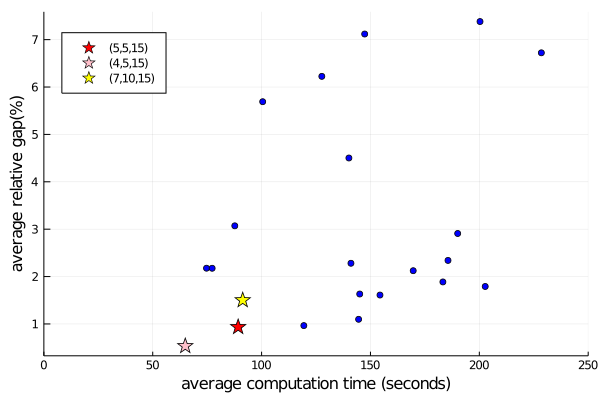}
         \caption{data set: 300N}
         \label{fig:kNN_para_tune_300bus_100}
     \end{subfigure}
     \hfill
    \begin{subfigure}[b]{0.45\textwidth}
         \centering
            \includegraphics[width=\textwidth]{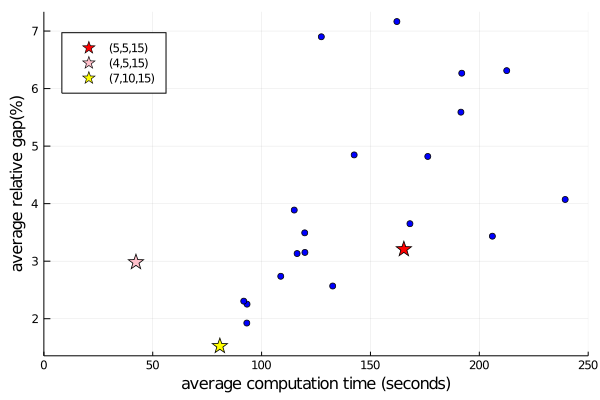}
         \caption{data set: 300H}
         \label{fig:kNN_para_tune_300bus_105}
    \end{subfigure}
    \caption{parameter tuning for kNN simulation method}
    \label{fig:kNN_tune}
\end{figure}  

\subsection{Computational results}
In this section, we report the computational results of the proposed methods. For the sake of accuracy, we here consider a cardinality constrained OTS problem, i.e., the formulation \ref{eq:DC_OTS} with the addition of an upper bound on the number of edges that can be switched off. In formula $\sum_{(i,j) \in E} (1-x_{ij}) \leq L$. For the experiments herein presented, we set the the upper bound $L$  to  45.

In Table \ref{tab:comp_results_demand}, we report a summary of the computational performances of the proposed approaches using the ``best" parameters' setting displayed in Table \ref{tab:para}. For comparison purposes, we also display the performances of the LWP method. All the displayed values are average values computed on the set of twenty instances in the specific class. For each method, Table\ref{tab:comp_results_demand} summarizes the following statistics: 
\begin{itemize}
\item "rel\_gap", i.e., the average relative gap;
\item "time", i.e., the average solution time in seconds;
\item "$\#$ u", i.e., the number of instances not solved to optimality within the time-limit; 
\item "opt\_gap", i.e., the average optimality gap. 
\end{itemize}
Please note that "rel\_gap" is not displayed for the LWP method because it provides the reference point and is zero by definition. Character "\_" denote the case where no observation is available. The best statistic value over the three method is highlighted with bold fonts for each class of instances. By and large, the kSP method is outperforming the other approaches. It solved all the instances at optimality, with short computational times (being the fastest on five classes of instances out of six), and without compromising the solution quality. Indeed, it provides better solutions than LWP for two classes of instances and for two others it provides solutions of the same quality. The average speedup of kSP over all 120 instances is almost 13 times with respect to LWP with a peak of 39.8 on the most difficult to solve instances (e.g., class 118H). 
The kNN shows very good computational performances on the classes of instances derived from case118B, especially in terms of solutions' quality. However, on instances of the case300B type, we didn't get results of the same quality, likely due to a not fully calibrated tuning phase. Indeed, if we were equipped with an oracle that returns the best parameters' setting for each instance we could achieve much better performance, as demonstrated by the statistics displayed in Table \ref{tab:comp_results_demand_best_para_each_instance}. This is valid for both the methods. For the sake of clarity, for each instance the best parameters' setting has been identified as follows:
first, we  filter out all the solutions whose objective value is greater than the smallest objective value obtained for the specific class increased by 0.1\%. From the remaining solutions, we further filter out all the solutions whose computation time is longer than the minimum computation time plus 1 second. Finally, if there is more than one solution remaining, we choose the one with the best objective function value. This table further highlight the potentials of the proposed approaches.

\begin{sidewaystable}
    \begin{center}
        \input{table_comparison}
    \end{center}
\end{sidewaystable}

\subsection{Computational analysis of the cardinality constraints.}

In practice, power system operators may want to limit the option of switching off branches of the grid for system’s reliability reasons. We here further analyze the impact of bounding the number of branches that can be feasibly switched-off on the computational performances of the proposed methods. Because these bounds may have considerable influence on the DC-OTS solution times, it is desirable that the proposed methods scale well with this limit. 
For this analysis, we only consider the nominal scenarios, i.e., the classes of instances  118N and 300N.
Our findings confirm the results of  \cite{fattahi2018bound, kocuk2016} about the increased computational challenge of solving the DC-OTS problem with a restricted number of edges that can be feasibly switched off. Indeed, smaller is the parameter $L$ (i.e., the number of edges that can be switched off) and more difficult are the instances to be solved.  The results are summarized in Table \ref{tab:comp_results_limits_118bus} and \ref{tab:comp_results_limits_300bus}.

kSP outperforms LWP in both quality of solutions - with a negative relative gap in seven out of eight cases across the two classes of instances considered - and  computational time. The more challenging it is the instance of the problem, the larger it is the potential benefit of using kSP. kNN finds more difficulties  in taming the complexity of the problem, although for some of the cases considered it computes good quality solution with negative relative gaps.

\begin{sidewaystable}
\bigskip\bigskip  
\smallskip
\begin{center}
\begin{tabular}{c|cccc|cccc|ccc}
\hline
& \multicolumn{4}{c}{kSP} & \multicolumn{4}{c}{kNN simulation} & 
\multicolumn{3}{c}{lwp}\\
\hline
L  & rel\_gap & time & \# u & opt\_gap & rel\_gap & time & \# u & opt\_gap & time & \# u & opt\_gap \\
\hline
10          & \textbf{0.04\%}          & \textbf{63.7}  & \textbf{1}    & 0.48\%   & -0.03\%  & 412.9     & 9             & 2.78\%    & 407.9     & 11            & 6.50\%   \\
15          & \textbf{-0.06\%} & \textbf{220.6} & \textbf{6}    & 2.26\%   & 0.03\%   & 350.1     & 10            & 5.22\%    & 345.1     & 9             & 4.06\%   \\
20          & \textbf{-0.11\%} & \textbf{65.2}  & \textbf{1}    & 0.51\%   & -0.04\%  & 275.6     & 8             & 2.99\%      & 212.5     & 6             & 3.18\%   \\
25          & \textbf{-0.04\%} & \textbf{41.3}  & \textbf{1}    & 0.26\%   & -0.01\%  & 296.2     & 8             & 2.22\%     & 177.5     & 4             & 1.99\%         \\ \hline     
\end{tabular}
\caption{Computational comparison across different values of L for the 118N class of instances.}
\label{tab:comp_results_limits_118bus}

\bigskip\bigskip
\smallskip
\begin{tabular}{c|cccc|cccc|ccc}
\hline
& \multicolumn{4}{c}{kSP} & \multicolumn{4}{c}{kNN simulation} & 
\multicolumn{3}{c}{lwp}\\
\hline
L  & rel\_gap & time & \# u & opt\_gap & rel\_gap & time & \# u & opt\_gap & time & \# u & opt\_gap \\
\hline
10 & \textbf{-0.77\%} & 600.0 & 20 & 23.3\% & -0.51\% & \textbf{573.3} & 19  & 27.00\% & 600.0 & 20 & 118\% \\
15 & \textbf{-0.16\%} & 600.0 & 20 & 6.7\%  & 0.44\%  & 600.0 & 20 & 10.43\% & 600.0 & 20 & 37\%  \\
20 & \textbf{-0.09\%} & 180.4 & 2  & 0.3\%  & 0.52\%  & 411.1 & 1  & 5.27\% & 473.5 & 12 & 29\%  \\
25 & \textbf{-0.01\%} & 86.1  & 1  & 0.1\%  & 0.53\%  & 245.1  & 0  & 0.63\% & 384.9 & 7  & 3\%  \\ \hline     
\end{tabular}
\caption{Computational comparison across different values of L for the 300N class of instances.}
\label{tab:comp_results_limits_300bus}
\end{center}
\end{sidewaystable}

\section{Conclusion and Further Work}
\label{sec::Conc}
In this paper, we present two data-driven heuristic methods to compute big M values with the purpose of tightening the mathematical programming formulation of the DC-OTS problem. Having small enough big M values, it may improve the performance of optimization solvers without scarifying the solution quality. The methods herein proposed exploit information on power demands and generation costs. The computational experience on a set of 120 instances -with different features in terms of power demand and network structure - demonstrates the viability of the proposed approaches. More in particular, the kSP method gives competitive solutions in short computational times with respect to the current approach, which is based on computing big M values with the longest-weighted path. Although, not at the same level of effectiveness, the kNN method also shows good performances especially on certain classes of instances. The relatively superior performance of kSP method can be attributed to the fact that it uses problem-specific knowledge. 
We also note that  kSP is quit robust with respect to parameters' setting, thus making this method easy to implement and resilient to demand forecasting errors. On the other hand, kNN seems more sensitive to parameters' setting. It is important to investigate strategies that provide instance specific settings to achieve the full potential of the proposed approaches. We also verified that both kSP and kNN simulation method provide better solutions - with an improvement of over $20\%$ - than those computed by the state of the art heuristic method \cite{fuller2012fast}. Moreover it will be of interest to consider a more accurate model including unit commitment constraints and multiple time periods adapting the proposed heuristics.

\section*{Acknowledgement}
We would like to thank to Prof. Adam Letchford for his comments and suggestions on a preliminary draft of the paper.

\clearpage
\bibliographystyle{plain}
\bibliography{OTS_paper}
\clearpage

\begin{appendices}

\section{Computational Results on Relative Gaps and Big-M Values}
Some information on relative gaps are shown in Table \ref{tab:rel_gap_demand}, \ref{tab:rel_gap_limits_118bus} and \ref{tab:rel_gap_limits_300bus}. "max\_gap", "min\_gap", "avg\_gap", "std\_gap" denote maximum, minimum, average and standard deviation of relative gaps, respectively. "\# non\_neg" denotes the number of instances that have a relative gap no smaller than $0.001\%$. Note that larger value indicates inferior performance.

Some information on big-M values is shown in Table \ref{tab:bigM_demand}. We select one instance in each data set and calculate the ratio between big M parameter computed by our proposed method and big M parameter computed by the longest path method for each edge in that instance. "max\_ratio", "avg\_ratio", "min\_ratio" and "std\_ratio" denote the maximum, average, minimum and standard deviation of those ratios, respectively.

\begin{sidewaystable}
\smallskip
\begin{tabular}{c|ccccc|ccccc}
\hline
& \multicolumn{5}{c}{kSP} & \multicolumn{5}{c}{kNN simulation} \\
\hline
 CoI & max\_gap & min\_gap & avg\_gap & std\_gap & \# non\_neg & max\_gap & min\_gap & avg\_gap & std\_gap & \# non\_neg \\
 \hline 
118L                                                & 0.09\%   & -0.25\%  & 0.00\%   & 0.07\%   & 12          & 0.09\%   & -0.26\%  & 0.00\%   & 0.07\%   & 9           \\
118N                                                 & 0.07\%   & -0.21\%  & -0.01\%  & 0.06\%   & 10          & 0.17\%   & -0.13\%  & 0.00\%   & 0.07\%   & 12         \\
118H                                                & 1.50\%   & -1.52\%  & 0.03\%   & 0.58\%   & 8           & 0.28\% & -1.52\% & -0.09\% & 0.36\% & 9          \\
300L                                                & 0.01\%   & -0.04\%  & -0.01\%  & 0.01\%   & 5          & 4.93\%   & -0.02\%  & 0.74\%   & 1.06\%   & 17          \\
300N                                                 & 0.05\%   & -0.07\%  & 0.00\%   & 0.03\%   & 10          &1.50\% & -0.02\% & 0.53\% & 0.47\% & 18          \\
300H                                                & 1.67\%   & -0.07\%  & 0.16\%   & 0.46\%   & 11          &6.72\% & -0.01\% & 1.52\% & 2.25\% & 17   \\
\hline
\end{tabular}
\caption{Relative Gaps under Different Power Demands}
\label{tab:rel_gap_demand}

\bigskip  

\begin{tabular}{c|ccccc|ccccc}
\hline
& \multicolumn{5}{c}{kSP} & \multicolumn{5}{c}{kNN simulation} \\
\hline
 \# off & max\_gap & min\_gap & avg\_gap & std\_gap & \# non\_neg & max\_gap & min\_gap & avg\_gap & std\_gap & \# non\_neg \\
 \hline
10      & 0.04\%   & -0.56\%  & -0.04\%  & 0.14\%   & 7         &0.08\% & -0.32\% & -0.03\% & 0.10\% & 1           \\
15                                                    & 0.32\%   & -0.50\%  & -0.06\%  & 0.21\%   & 8           & 0.77\% & -0.36\% & 0.03\%  & 0.24\% & 11           \\
20                                                    & 0.04\%   & -0.81\%  & -0.11\%  & 0.20\%   & 3          & 0.34\% & -0.73\% & -0.04\% & 0.23\% & 7            \\
25                                                    & 0.09\%   & -0.53\%  & -0.04\%  & 0.15\%   & 6           & 0.20\% & -0.32\% & -0.01\% & 0.13\% & 12    \\ \hline    
\end{tabular}
\caption{Relative Gaps on Data Set 118N with Different Limits on Number of Switched-off Edges}
\label{tab:rel_gap_limits_118bus}

\bigskip  

\begin{tabular}{c|ccccc|ccccc}
\hline
& \multicolumn{5}{c}{kSP} & \multicolumn{5}{c}{kNN simulation} \\
\hline
 \# off & max\_gap & min\_gap & avg\_gap & std\_gap & \# non\_neg & max\_gap & min\_gap & avg\_gap & std\_gap & \# non\_neg \\
 \hline
10                                                    & 0.00\%   & -5.31\%  & -0.77\%  & 1.50\%   & 0           &4.24\% & -5.31\% & -0.51\% & 1.88\% & 2           \\
15                                                    & 0.09\%   & -1.02\%  & -0.16\%  & 0.24\%   & 2           & 1.56\% & -0.71\% & 0.44\%  & 0.54\% & 15          \\
20                                                    & 0.03\%   & -0.74\%  & -0.09\%  & 0.17\%   & 5           & 1.43\% & -0.63\% & 0.52\%  & 0.55\% & 18          \\
25                                                    & 0.08\%   & -0.10\%  & -0.01\%  & 0.04\%   & 6           & 1.43\% & -0.03\% & 0.53\%  & 0.46\% & 19             \\ \hline    
\end{tabular}
\caption{Relative Gaps on Data Set 300N with Different Limits on Number of Switched-off Edges}
\label{tab:rel_gap_limits_300bus}

\bigskip

\begin{tabular}{c|cccc|cccc}
\hline
& \multicolumn{4}{c}{kSP} & \multicolumn{4}{c}{kNN simulation} \\
\hline
 CoI & max\_ratio & avg\_ratio & min\_ratio & std\_ratio & max\_ratio & avg\_ratio & min\_ratio & std\_ratio \\
 \hline
118L                                             & 100.00\%   & 24.38\%    & 1.91\%     & 33.65\%    & 100.00\%   & 31.92\%    & 0.26\%     & 31.67\%    \\
118N                                          & 100.00\%   & 24.30\%    & 1.91\%     & 33.70\%    & 100.00\%   & 32.09\%    & 0.28\%     & 31.85\%    \\
118H                                            & 100.00\%   & 24.20\%    & 1.50\%     & 33.76\%    & 100.00\%   & 37.89\%    & 0.18\%     & 41.58\%    \\
300L                                            & 100.00\%   & 42.80\%    & 0.20\%     & 45.09\%    & 100.00\%   & 27.29\%    & 0.03\%     & 38.99\%    \\
300N                                      & 100.00\%   & 42.15\%    & 1.11\%     & 45.21\%    & 100.00\%   & 75.80\%    & 0.01\%     & 41.35\%    \\
300H                                         & 100.00\%   & 41.98\%    & 1.11\%     & 45.22\%    & 100.00\%   & 81.57\%    & 0.003\%    & 37.27\%      \\
\hline
\end{tabular}
\caption{Big-M Values under Different Power Demands}
\label{tab:bigM_demand}
\end{sidewaystable}

\end{appendices}

\end{document}

%% file: table_comparison.tex
\begin{tabular}{c|ccc|cccc|ccc}
\hline
& \multicolumn{3}{c}{kSP} & \multicolumn{4}{c}{kNN simulation} & 
\multicolumn{3}{c}{lwp}\\
\hline
CoI  & rel\_gap & time & \# u  & rel\_gap & time & \# u & opt\_gap  & time & \# u & opt\_gap \\
\hline
118L & 0.00\%           & \textbf{6.5}  & \textbf{0}       & 0.00\%           & 73.9      & 1             & 0.41\%     & 126.1     & 4             & 0.35\%   \\
118N    & \textbf{-0.01\%} & \textbf{12.9} & \textbf{0}      & 0.00\%           & 123.7     & 3             & 0.24\%     & 138.9     & 4             & 0.23\%   \\ 
118H & 0.03\%         &\textbf{5.2}  & 0                 & \textbf{-0.09\%} & 90.9     & 1             & 0.11\%          & 207.1     & 6             & 0.48\%   \\
300L  & \textbf{-0.01\%}           & 8.5     & 0       & 0.74\%         & \textbf{5.5} & 0           & \_         & 16.4      & 0             & \_       \\
300N    & 0.00\%           & \textbf{64.2}          & 0                    & 0.53\%           & 64.9       & 1   & 0.39\%    & 204.4     & 5     & 0.16\%   \\
300H & \textbf{-0.01\%}    & \textbf{50.5}         & 0       & 1.5\%           & 80.8  & 1           & \_         & 80.0      & 0             & \_       \\ \hline     
\end{tabular}
\caption{Computational results with ``best" parameters' setting per class of instances.}
\label{tab:comp_results_demand}

\bigskip
\begin{tabular}{c|ccc|ccc|ccc}
\hline
& \multicolumn{3}{c}{kSP} & \multicolumn{3}{c}{kNN simulation} & 
\multicolumn{3}{c}{lwp}\\
\hline
CoI  & rel\_gap & time & \# u  & rel\_gap & time & \# u   & time & \# u & opt\_gap \\
\hline
118L & \textbf{-0.01\%}     & \textbf{2.6}   & 0         & 0.00\%      & 13.5      & 0            & 126.1     & 4             & 0.35\%   \\
118N    & \textbf{-0.03\%} &\textbf{8.9}  & 0  & -0.02\% & 17.5  & 0       & 138.9     & 4             & 0.23\%   \\ 
118H & 0.02\%  & \textbf{2.9}  & 0  & \textbf{-0.10\%} & 23.1  & 0                 & 207.1     & 6             & 0.48\%   \\
300L & \textbf{-0.01\%} & \textbf{4.6}  & 0 & 0.04\%  & 7.0   & 0          & 16.4      & 0             & \_       \\
300N  & 0.02\%  & \textbf{13.5} & 0  & 0.12\%  & 105.2 & 0       & 204.4     & 5     & 0.16\%   \\
300H & \textbf{-0.01\%} & \textbf{17.1} & 0  & 0.16\%  & 32.6  & 0          & 80.0      & 0             & \_       \\ \hline     
\end{tabular}
\caption{Computational results with instance specific parameters' setting.}
\label{tab:comp_results_demand_best_para_each_instance}